\documentclass{article}
\usepackage[latin1]{inputenc}
\usepackage{amsmath}
\usepackage{amsthm,amssymb,amsfonts}
\usepackage{graphicx}
\usepackage{graphics}
\usepackage{verbatim}
\usepackage{color,cite}
\usepackage{latexsym}
\usepackage{lscape}
\usepackage[T1]{fontenc}
\usepackage[all,cmtip]{xy}
\usepackage{tikz}
\usetikzlibrary{arrows, decorations.pathmorphing, positioning, backgrounds, fit, petri}
\usetikzlibrary{patterns}
\usepackage{xspace}
\usepackage{enumitem}
\usepackage{xcolor}

\theoremstyle{plain}
\newtheorem{theorem}{Theorem}[section]
\newtheorem{lemma}[theorem]{Lemma}
\newtheorem{proposition}[theorem]{Proposition}
\newtheorem{corollary}[theorem]{Corollary}

\newtheorem*{definition}{Definition}

\newtheorem*{remark}{Remark}

\newtheorem*{theorem:MainThm}{Theorem \ref{theorem:MainThm}}
\newtheorem*{corollary:MainCor}{Corollary \ref{corollary:MainCor}}

\begin{document}

\title{Affine Invariant Submanifolds with Completely Degenerate Kontsevich-Zorich Spectrum}
\author{David Aulicino\thanks{This material is based upon work supported by the National Science Foundation under Award No. DMS - 1204414.}}
\date{}

\newcommand{\splin}{$\text{SL}_2(\mathbb{R})$}
\newcommand{\spolin}{$\text{SO}_2(\mathbb{R})$}
\newcommand{\RankOne}{$\mathcal{D}_g (1)$}
\newcommand{\RankOneF}{$\mathcal{D}_5 (1)$}
\newcommand{\nc}{\newcommand}

\nc\bB{\mathbb{B}}
\nc\bC{\mathbb{C}}
\nc\bD{\mathbb{D}}
\nc\bE{\mathbb{E}}
\nc\bF{\mathbb{F}}
\nc\bG{\mathbb{G}}
\nc\bH{\mathbb{H}}
\nc\bI{\mathbb{I}}
\nc{\bJ}{\mathbb{J}}
\nc\bK{\mathbb{K}}
\nc\bL{\mathbb{L}}
\nc\bM{\mathbb{M}}
\nc\bN{\mathbb{N}}
\nc\bO{\mathbb{O}}
\nc\bP{\mathbb{P}}
\nc\bQ{\mathbb{Q}}
\nc\bR{\mathbb{R}}
\nc\bS{\mathbb{S}}
\nc\bT{\mathbb{T}}
\nc\bU{\mathbb{U}}
\nc\bV{\mathbb{V}}
\nc\bW{\mathbb{W}}
\nc\bY{\mathbb{Y}}
\nc\bX{\mathbb{X}}
\nc\bZ{\mathbb{Z}}

\nc\cA{\mathcal{A}}
\nc\cB{\mathcal{B}}
\nc\cC{\mathcal{C}}
\nc\cD{\mathcal{D}}
\nc\cE{\mathcal{E}}
\nc\cF{\mathcal{F}}
\nc\cG{\mathcal{G}}
\nc\cH{\mathcal{H}}
\nc\cI{\mathcal{I}}
\nc{\cJ}{\mathcal{J}}
\nc\cK{\mathcal{K}}
\nc\cM{\mathcal{M}}
\nc\cN{\mathcal{N}}
\nc\cO{\mathcal{O}}
\nc\cP{\mathcal{P}}
\nc\cQ{\mathcal{Q}}
\nc\cS{\mathcal{S}}
\nc\cT{\mathcal{T}}
\nc\cU{\mathcal{U}}
\nc\cV{\mathcal{V}}
\nc\cW{\mathcal{W}}
\nc\cY{\mathcal{Y}}
\nc\cX{\mathcal{X}}
\nc\cZ{\mathcal{Z}}

\maketitle



\begin{abstract}
We prove that if the Lyapunov spectrum of the Kontsevich-Zorich cocycle over an affine \splin -invariant submanifold is completely degenerate, i.e. $\lambda_2 = \cdots = \lambda_g = 0$, then the submanifold must be an arithmetic Teichm\"uller curve in the moduli space of Abelian differentials over surfaces of genus three, four, or five.  As a corollary, we prove that there are at most finitely many such Teichm\"uller curves.
\end{abstract}

\section{Introduction}

The Lyapunov exponents of the Kontsevich-Zorich cocycle give information about the dynamics of numerous systems such as billiards in polygons with rational angles, interval exchange transformations, and the Teichm\"uller geodesic flow on the moduli space of Abelian differentials on Riemann surfaces.  More precisely, we consider the Lyapunov exponents of the Kontsevich-Zorich cocycle on the absolute cohomology bundle over affine \splin -invariant submanifolds of the moduli space of Abelian differentials on genus $g$ surfaces.  These exponents were studied extensively in \cite{ForniDev}, \cite{AvilaVianaSimp}, and \cite{EskinKontsevichZorich2}.  In \cite{ForniDev}, it was proven that the smallest non-negative exponent is always positive with respect to the canonical measures on strata of the moduli space of Abelian differentials.  In \cite{AvilaVianaSimp}, the Kontsevich-Zorich conjecture was verified using techniques independent of \cite{ForniDev}, i.e. the spectrum of $2g$ exponents of the cocycle is simple with respect to the canonical measures on strata.  Explicit formulas for sums of the positive Lyapunov exponents of the Kontsevich-Zorich cocycle were given in \cite{EskinKontsevichZorich2}.  On the contrary, there are examples of orbit closures where the Lyapunov exponents are zero.  An affine invariant submanifold with $\lambda_2 = \cdots = \lambda_g = 0$ is said to have completely degenerate KZ-spectrum.  In this paper, we prove:

\begin{theorem:MainThm}
If $\mathcal{M}$ is an affine invariant submanifold with completely degenerate KZ-spectrum, then $\mathcal{M}$ is an arithmetic Teichm\"uller curve.
\end{theorem:MainThm}

Since rational billiards do not necessarily have generic orbit closures, it is important to understand all orbit closures in order to understand rational billiards.  From this perspective, the study of which orbit closures admit zero Lyapunov exponents is essential to understanding the dynamics on a specific rational billiard table.

A relation between the geometry of the Hodge bundle and the Lyapunov exponents has been known for some time.  Specifically, a connection between the curvature of the Hodge bundle and the neutral Oseledec bundle was studied in \cite{ForniMatheusZorichLyapSpectHodge}.  In \cite{ForniMatheusZorichZeroLyapExpsHodge} it was conjectured that there are exactly two mechanisms that produce zero Lyapunov exponents.  The Forni subspace mechanism defined below will be the only one relevant to this paper.

In \cite{ForniHand}, Forni gives an example of a genus three square-tiled surface generating a Teichm\"uller curve with completely degenerate KZ-spectrum.  In \cite{ForniMatheusZorichSqTiled}, Forni, Matheus, and Zorich found an example of a genus four square-tiled surface generating a Teichm\"uller curve that also has completely degenerate KZ-spectrum, and proved that these two examples are the only square-tiled \emph{cyclic} covers with completely degenerate KZ-spectrum in any genus.  M\"oller \cite{MollerShimuraTeich} found that these are the only examples of Teichm\"uller curves with completely degenerate KZ-spectrum with possible exceptions in the moduli space of Abelian differentials on genus five surfaces.  Without assuming any of the structure of orbit closures established in \cite{EskinMirzakhaniInvariantMeas} and \cite{EskinMirzakhaniMohammadiOrbitClosures} parts of Theorem \ref{theorem:MainThm} were proven for low genus in \cite{AulicinoCompDegKZ}.  In \cite{EskinKontsevichZorich2}, it was shown that there are no such regular affine \splin -invariant submanifolds for $g \geq 7$.  Since then, \cite{AvilaMatheusYoccozRegAIS} showed that every orbit closure is regular.

Combining Theorem \ref{theorem:MainThm} with the results recalled above implies that the only affine invariant submanifolds with completely degenerate KZ-spectrum are the two known Teichm\"uller curves in genus three and four, and any other such affine invariant submanifold must be an arithmetic Teichm\"uller curve in genus five.  Furthermore, the proof of Theorem \ref{theorem:MainThm} yields the following corollary.

\begin{corollary:MainCor}
There are at most finitely many Teichm\"uller curves with completely degenerate KZ-spectrum.
\end{corollary:MainCor}

The results in this paper rely on the recent fundamental work of Eskin and Mirzakhani \cite{EskinMirzakhaniInvariantMeas}, and Eskin, Mirzakhani, and Mohammadi \cite{EskinMirzakhaniMohammadiOrbitClosures}, which established that all \splin -orbit closures are affine \splin -invariant submanifolds.  Furthermore, the results of Avila, Eskin, and M\"oller \cite{AvilaEskinMollerForniBundle}, Wright \cite{WrightFieldofDef}, and M\"oller \cite{MollerShimuraTeich} are essential ingredients in the proof of Theorem \ref{theorem:MainThm}.

\

\noindent \textbf{Acknowledgments:}  The author would like to thank Alex Eskin for introducing him to this project and for explaining many of the ideas involved in it.  The author is also very grateful to Alex Eskin for the generosity of his time and patience throughout the discussions, and for carefully reading earlier drafts of this manuscript.  The author would also like to thank Martin M\"oller for contributing the strategy for the proof of Theorem \ref{theorem:MainThm}.  The author is grateful to Alex Wright for the careful explanations of his work and for many helpful discussions.

\section{Preliminaries}

\noindent \textbf{Strata of Abelian Differentials:}  Let $X$ be a Riemann surface of genus $g \geq 2$ carrying an Abelian differential $\omega$.  If $\omega$ is holomorphic, then it has zeros of total order $2g-2$.  Let $\kappa$ denote a partition of $2g-2$.  We consider the set of all pairs $(X,\omega)$, where the orders of the zeros of $\omega$ are prescribed by $\kappa$.  This set, denoted $\mathcal{H}(\kappa)$ is called a \emph{stratum of Abelian differentials}.  We assume throughout that $(X,\omega)$ has unit area with respect to the area form induced by $\omega$.

\

\noindent \textbf{\splin ~Action:} An Abelian differential $\omega$ induces a flat structure on $X$.  There is a natural action by \splin ~on this structure defined by embedding its charts in $\mathbb{R}^2$ and multiplying by an element of \splin .  Furthermore, the action preserves the area of $X$ with respect to $\omega$.  The action by the subgroup of diagonal matrices is called the \emph{Teichm\"uller geodesic flow}.

\

\noindent \textbf{Period Coordinates:} Let $\Sigma$ denote the set of zeros of $\omega$.  If we fix a basis $\{\gamma_1, \ldots, \gamma_k\}$ for $H_1(X,\Sigma, \mathbb{Z})$, then we get a local map into relative cohomology
$$\Phi(X,\omega) := \left( \int_{\gamma_1} \omega, \ldots, \int_{\gamma_k} \omega \right) \in H^1(X,\Sigma, \mathbb{C}).$$

\

\noindent \textbf{Orbit Closures:} It was proven in \cite{EskinMirzakhaniInvariantMeas} and \cite{EskinMirzakhaniMohammadiOrbitClosures}, that the closure of every \splin ~orbit in $\mathcal{H}(\kappa)$ is an affine \splin -invariant submanifold $\mathcal{M}$ that admits a finite \splin -invariant measure $\nu$, which is affine with respect to period coordinates.  For simplicity, we abbreviate the term \emph{affine \splin -invariant submanifold} by \emph{AIS}.

The tangent space of $\mathcal{M}$ can be given in period coordinates as a subspace $T_{\mathbb{C}}(\mathcal{M})$.  It satisfies $T_{\mathbb{C}}(\mathcal{M}) = \mathbb{C} \otimes T_{\mathbb{R}}(\mathcal{M})$, where $T_{\mathbb{R}}(\mathcal{M}) \subset H^1(X,\Sigma,\mathbb{R})$.  Let $p: H^1(X,\Sigma,\mathbb{R}) \rightarrow H^1(X,\mathbb{R})$ be the natural projection to absolute cohomology.

If $\dim T_{\mathbb{R}}(\mathcal{M}) = 2$, then $\mathcal{M}$ is called a \emph{Teichm\"uller curve} and any surface $(X,\omega)$ generating a Teichm\"uller curve is called a \emph{Veech surface}.  (This is not the definition of a Veech surface, but a theorem of John Smillie.)  If the Veech surface is square-tiled, then it generates an arithmetic Teichm\"uller curve.

\

\noindent \textbf{Lyapunov Exponents:} The bundle $H^1_{\mathbb{F}}$ over $\mathcal{H}(\kappa)$ is the bundle with fibers $H^1(X,\mathbb{F})$ and a flat connection given by identifying nearby lattices $H^1(X,\mathbb{Z})$ and $H^1(X',\mathbb{Z})$.  If $\mathcal{M}$ is an AIS, then the Teichm\"uller geodesic flow acts on every point in $\mathcal{M}$ and thus, induces a flow on $H^1_{\mathbb{R}}$.  This flow is known as the \emph{Kontsevich-Zorich cocycle} (KZ-cocycle).

If we consider orbits under the Teichm\"uller geodesic flow that return infinitely many times to a neighborhood of the starting point, then it is possible to compute the monodromy matrix $A(t)$ at each return time $t$.  By computing the logarithms of the eigenvalues of the $A(t)A^T(t)$, normalizing them by twice the length of the geodesic at time $t$, and letting $t$ tend to infinity, we get a collection of $2g$ numbers known as the \emph{spectrum of Lyapunov exponents of the KZ-cocycle}.  By the Oseledec multiplicative ergodic theorem, these numbers will not depend on the initial starting point for $\nu$-almost every choice of initial data.  Since cohomology admits a symplectic basis, this spectrum is symmetric and therefore it suffices to consider the set of $g$ non-negative Lyapunov exponents of the KZ-cocycle
$$1 = \lambda_1^{\nu} \geq \cdots \geq \lambda_g^{\nu} \geq 0,$$
which will be known as the \emph{KZ-spectrum}.  We will suppress the measure from now on and always assume it to be the canonical measure guaranteed by \cite{EskinMirzakhaniInvariantMeas}.

If the Lyapunov exponents of the Kontsevich-Zorich cocycle over an AIS $\mathcal{M}$ satisfy
$$1 = \lambda_1 > \lambda_2 = \cdots = \lambda_g = 0,$$
then $\mathcal{M}$ has \emph{completely degenerate KZ-spectrum}.  The first such example in genus three was found in \cite{ForniHand} and another example in genus four by \cite{ForniMatheusZorichSqTiled}.  Further results on this question nearing a complete classification of all such AIS with completely degenerate KZ-spectrum were established in \cite{MollerShimuraTeich}, \cite{EskinKontsevichZorich2}, and \cite{AulicinoCompDegKZ}.

\

\noindent \textbf{Forni Subspace:} The \emph{Forni subspace} $F(x) \subset H^1(X,\mathbb{R})$ was formally defined in \cite{AvilaEskinMollerForniBundle}.  The subspace $F(x)$ is the maximal \splin -invariant subspace on which the KZ-cocycle acts by isometries with respect to the Hodge inner product.  It was proven in \cite[Thm. 1.3]{AvilaEskinMollerForniBundle} that for $\nu$-a.a. $x$, $p(T_{\mathbb{R}}(\mathcal{M}))(x)$ is orthogonal to $F(x)$ with respect to the Hodge inner product.

\

\noindent \textbf{Rank One Locus:} Let $X$ have genus $g$.  Let $\{\theta_1, \ldots, \theta_g\}$ be a basis of Abelian differentials on $X$.  Define the $ij$-component of the \emph{derivative of the period matrix} at $X$ in direction $\mu_{\omega}$ by the Ahlfors-Rauch variational formula
$$d\Pi(X, \omega) := \left(\frac{d\Pi}{\mu_{\omega}}\right)_{ij} = \int_X \theta_i\theta_j\frac{\overline{\omega}}{\omega}.$$
The \emph{rank one locus} is the set
$$\mathcal{D}_g(1) = \{(X,\omega) | \text{Rank}(d\Pi/\mu_{\omega}) = 1\}.$$

\

The problem of studying Teichm\"uller discs with completely degenerate KZ-spectrum was addressed from the perspective of the rank one locus in \cite{ForniHand} and \cite{AulicinoCompDegKZ}.  In general, if the dimension of the Forni subspace is $f$, then the value $k$ of the rank $k$ locus is related to $f$ by the formula $f/2 + k = g$, e.g. the rank one locus corresponds to the Forni subspace having maximal dimension $2g-2$.  Note that the rank $k$ locus is a subset of the moduli space of Abelian differentials that is \emph{not} known to be \splin -invariant, whereas the Forni subspace is a subbundle of $H^1_{\mathbb{R}}$ that is \splin -invariant by definition.  

We reconcile these perspectives in the following lemma.  The equivalence between $(1)$ and $(2)$ follows by definition of the Forni subspace.  The equivalence between $(2)$ and $(3)$ follows by definition in one direction, and by observing that $\mathcal{D}_g(1)$ is a closed set in the other direction.  Alternatively, if the KZ-spectrum is completely degenerate, then the Forni-Kontsevich formula \cite[Cor. 5.3]{ForniDev} forces the Forni subspace to have maximal dimension.

\begin{lemma}
\label{MaxForniSp}
The following are equivalent:
\begin{itemize}
\item[(1)] An AIS $\mathcal{M}$ has completely degenerate KZ-spectrum
\item[(2)] The Forni subspace of $\mathcal{M}$ has maximal dimension, i.e. for a.e. $x \in \mathcal{M}$, $\dim_{\mathbb{R}}(F(x)) = 2g-2$
\item[(3)] $\mathcal{M} \subset \mathcal{D}_g(1)$.
\end{itemize}
\end{lemma}

\

\noindent \textbf{Field of Definition:} In \cite{WrightFieldofDef}, the field of definition of an AIS was introduced.  The \emph{field of definition}, $\textbf{k}(\mathcal{M})$ of an AIS $\mathcal{M}$ is the smallest subfield of $\mathbb{R}$ such that $\mathcal{M}$ can be defined in local period coordinates by linear equations in $\textbf{k}(\mathcal{M})$.  It was proven that this is well-defined for every AIS and has degree at most $g$ over $\mathbb{Q}$, \cite[Thm. 1.1]{WrightFieldofDef}.

\

\noindent \textbf{Optimal Torus Coverings:}  In \cite{MollerShimuraTeich}, it was shown that every Teichm\"uller curve with completely degenerate KZ-spectrum is arithmetic, i.e. it is generated by a square-tiled surface.  Square-tiled surfaces $(X,\omega)$ naturally cover the torus $(\mathbb{T}^2, dz)$.  Let $\pi_{opt}: X \rightarrow \mathbb{T}^2$ denote the torus covering of minimal degree, i.e. $\pi_{opt}$ does not admit an intermediate torus cover.  It was shown in \cite{MollerShimuraTeich} that for arithmetic Teichm\"uller curves with completely degenerate KZ-spectrum the degree of $\pi_{opt}$ depends only on the stratum in which the Teichm\"uller curve lies.  Moreover, \cite{MollerShimuraTeich} calculated this degree $d_{opt}$ explicitly.

\

\noindent \textbf{The Eierlegende Wollmilchsau and the Ornithorynque:} We define the surfaces that generate the two known examples of Teichm\"uller curves with completely degenerate KZ-spectrum by depicting them as square-tiled surfaces in Figures \ref{EWFig} and \ref{OrniFig}.

\begin{figure}[htb]
\centering
\begin{minipage}{0.4\linewidth}
\centering
\begin{tikzpicture}[scale=0.50]
\draw (0,0) -- (0,2) -- (6,2)-- (6,4) -- (14,4) -- (14,2) -- (8,2) -- (8,0) -- cycle;
\draw (2,0) -- (2,2);
\draw (4,0) -- (4,2);
\draw (6,0) -- (6,2);
\draw (6,2) -- (8,2);
\draw (8,2) -- (8,4);
\draw (10,2) -- (10,4);
\draw (12,2) -- (12,4);
\foreach \x in {(0,0),(0,2),(2,0),(2,2),(4,0),(4,2),(6,0),(6,2),(6,4),(8,0),(8,2),(8,4),(10,2),(10,4),(12,2),(12,4),(14,2),(14,4)} \filldraw[fill=black] \x circle (3pt);
\draw (1,2) node[above] {\tiny A};
\draw (3,2) node[above] {\tiny B};
\draw (5,2) node[above] {\tiny C};
\draw (7,4) node[above] {\tiny E};
\draw (9,4) node[above] {\tiny D};
\draw (11,4) node[above] {\tiny G};
\draw (13,4) node[above] {\tiny F};
\draw (1,0) node[below] {\tiny D};
\draw (3,0) node[below] {\tiny E};
\draw (5,0) node[below] {\tiny F};
\draw (7,0) node[below] {\tiny G};
\draw (9,2) node[below] {\tiny C};
\draw (11,2) node[below] {\tiny B};
\draw (13,2) node[below] {\tiny A};
\end{tikzpicture}
\end{minipage}
\caption{The Eierlegende Wollmilchsau}
\label{EWFig}
\end{figure}
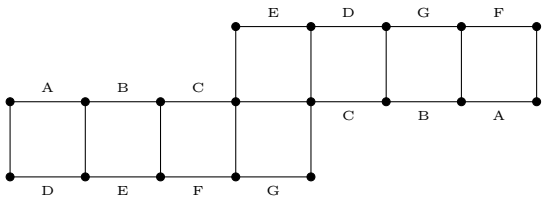

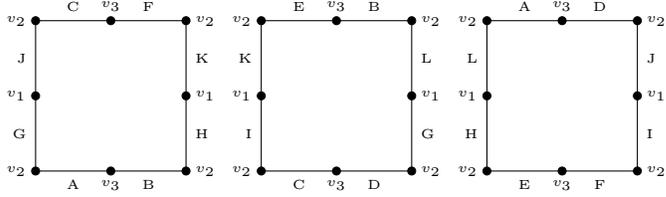
\begin{figure}[htb]
\centering
\begin{minipage}{0.4\linewidth}
\centering
\begin{tikzpicture}[scale=0.50]
\draw (-8,0) -- (-8,4) -- (-4,4) -- (-4,0) -- cycle;
\draw (-2,0) -- (-2,4) -- (2,4) -- (2,0) -- cycle;
\draw (4,0) -- (4,4) -- (8,4) -- (8,0) -- cycle;
\foreach \x in {(-8,0),(-8,2),(-8,4),(-6,0),(-6,4),(-4,4),(-4,2),(-4,0)} \filldraw[fill=black] \x circle (3pt);
\foreach \x in {(-2,0),(-2,2),(-2,4),(0,0),(0,4),(2,4),(2,2),(2,0)} \filldraw[fill=black] \x circle (3pt);
\foreach \x in {(4,0),(4,2),(4,4),(6,0),(6,4),(8,4),(8,2),(8,0)} \filldraw[fill=black] \x circle (3pt);
\draw (-8,0) node[left] {\tiny $v_2$};
\draw (-8,2) node[left] {\tiny $v_1$};
\draw (-8,4) node[left] {\tiny $v_2$};
\draw (-6,0) node[below] {\tiny $v_3$};
\draw (-6,4) node[above] {\tiny $v_3$};
\draw (-4,0) node[right] {\tiny $v_2$};
\draw (-4,2) node[right] {\tiny $v_1$};
\draw (-4,4) node[right] {\tiny $v_2$};

\draw (-2,0) node[left] {\tiny $v_2$};
\draw (-2,2) node[left] {\tiny $v_1$};
\draw (-2,4) node[left] {\tiny $v_2$};
\draw (0,0) node[below] {\tiny $v_3$};
\draw (0,4) node[above] {\tiny $v_3$};
\draw (2,0) node[right] {\tiny $v_2$};
\draw (2,2) node[right] {\tiny $v_1$};
\draw (2,4) node[right] {\tiny $v_2$};

\draw (4,0) node[left] {\tiny $v_2$};
\draw (4,2) node[left] {\tiny $v_1$};
\draw (4,4) node[left] {\tiny $v_2$};
\draw (6,0) node[below] {\tiny $v_3$};
\draw (6,4) node[above] {\tiny $v_3$};
\draw (8,0) node[right] {\tiny $v_2$};
\draw (8,2) node[right] {\tiny $v_1$};
\draw (8,4) node[right] {\tiny $v_2$};

\draw (-7,4) node[above] {\tiny C};
\draw (-5,4) node[above] {\tiny F};
\draw (-7,0) node[below] {\tiny A};
\draw (-5,0) node[below] {\tiny B};
\draw (-8,1) node[left] {\tiny G};
\draw (-8,3) node[left] {\tiny J};
\draw (-4,1) node[right] {\tiny H};
\draw (-4,3) node[right] {\tiny K};

\draw (-1,4) node[above] {\tiny E};
\draw (1,4) node[above] {\tiny B};
\draw (-1,0) node[below] {\tiny C};
\draw (1,0) node[below] {\tiny D};
\draw (-2,1) node[left] {\tiny I};
\draw (-2,3) node[left] {\tiny K};
\draw (2,1) node[right] {\tiny G};
\draw (2,3) node[right] {\tiny L};

\draw (5,4) node[above] {\tiny A};
\draw (7,4) node[above] {\tiny D};
\draw (5,0) node[below] {\tiny E};
\draw (7,0) node[below] {\tiny F};
\draw (4,1) node[left] {\tiny H};
\draw (4,3) node[left] {\tiny L};
\draw (8,1) node[right] {\tiny I};
\draw (8,3) node[right] {\tiny J};

\end{tikzpicture}
\end{minipage}
\caption{The Ornithorynque}
\label{OrniFig}
\end{figure}

\subsection{Cylinder Configuration}
\label{CylConfigHomBasis}

In this section we recall a result from \cite{AulicinoCompDegKZ}, which describes the cylinder decomposition of a surface whose AIS has completely degenerate KZ-spectrum.  Following the work of \cite{WrightCylDef}, both \cite[Thm. 5.5]{AulicinoCompDegKZ} and Corollary \ref{RankOneCylConfig} below can be proven independently of \cite{AulicinoCompDegKZ}.  (See the remark below for a complete alternate proof.)  Then we choose a convenient homology basis on this surface that will be used in Lemma \ref{RatFieldDef} below.  The claims of this section are summarized in Figure \ref{HomBasis}.

\begin{definition}
Given $(X,\omega)$, let $\mathcal{F}_{\theta}$ denote the vertical foliation of $(X,e^{i\theta}\omega)$.  If, for all $\theta \in \mathbb{R}$, the existence of a closed regular trajectory in $\mathcal{F}_{\theta}$ implies that every trajectory in $\mathcal{F}_{\theta}$ is closed, then $(X,\omega)$ is \emph{completely periodic}.
\end{definition}

\begin{remark}
This definition of ``completely periodic'' is consistent with that of \cite{AulicinoCompDegKZ} and \cite{CaltaVeechSurfsCompPerGen2}, but different from that of \cite{SmillieWeissMinSets}.
\end{remark}

It is proven in \cite[Thm. 5.5]{AulicinoCompDegKZ}, that any $(X,\omega)$ generating an AIS with completely degenerate KZ-spectrum, is completely periodic, though it was not phrased with this terminology because it was not available at the time.  The cylinder configuration described in Corollary \ref{RankOneCylConfig} is depicted in Figure \ref{HomBasis}.  In the specific case of Veech surfaces, this corollary is \cite[Lem. 5.3]{MollerShimuraTeich}.

\

\noindent \textbf{Cylinder Index Notation:}  Throughout this paper the indices on the cylinders are taken modulo $r$ so that we implicitly assume $C_{r+1} = C_1$.

\begin{corollary}[\cite{AulicinoCompDegKZ}]
\label{RankOneCylConfig}
Let $(X,\omega)$ generate an AIS with completely degenerate KZ-spectrum.  For each $\theta \in \mathbb{R}$ such that the vertical foliation of $(X, e^{i\theta}\omega)$ is periodic, $(X, e^{i\theta}\omega)$ decomposes into a union of cylinders $C_1, \ldots, C_r$ such that each saddle connection on the top of $C_i$ is identified to a saddle connection on the bottom of $C_{i+1}$ and vice versa, for all $1 \leq i \leq r$.  Furthermore, the circumference of $C_i$ equals the circumference of $C_j$, for all $i,j$.
\end{corollary}

The proof of Corollary \ref{RankOneCylConfig} follows from considering any periodic direction, which exists because $(X,\omega)$ is completely periodic, and degenerating the surface by pinching the core curves of every cylinder in that direction.  The configuration of the connected components of the resulting degenerate surface must be arranged in a cycle in order for the derivative of the period matrix to have rank one.  This yields the identification scheme in Corollary \ref{RankOneCylConfig}.  Furthermore, no two saddle connections on the top of $C_i$ (or the bottom of $C_{i+1}$) can be identified to each other without contradicting the orientability of a foliation transverse to the vertical foliation of $\omega$.

\begin{remark}
Alternatively, there is a simple proof of \cite[Thm. 5.5]{AulicinoCompDegKZ}.  By \cite{AvilaEskinMollerForniBundle}, $p(T(\mathcal{M}))$ is orthogonal to the Forni subspace, which implies that $p(T(\mathcal{M}))$ is $2$-dimensional.  By \cite[Thm. 1.5]{WrightCylDef}, every translation surface in $\mathcal{M}$ is completely periodic.

By the Forni Criterion \cite{ForniCriterion}, the number of positive Lyapunov exponents is bounded below by the dimension of the core curves of the cylinders in homology.  If there is exactly one positive Lyapunov exponent, then the Forni Criterion implies that in each periodic direction every cylinder is homologous to every other cylinder.  Hence, Corollary \ref{RankOneCylConfig} follows.
\end{remark}

Finally, we introduce a basis $\mathcal{B}(X,\omega) \subset H_1(X,\mathbb{Z})$, which is partially depicted in Figure \ref{HomBasis}, for the homology space of $X$.  This basis is chosen in a way that is dependent on $\omega$, and will be abbreviated $\mathcal{B}$ when $(X,\omega)$ is understood.  The cycle $a_1$ represents the family of core curves of the cylinders, which are pairwise homologous by Corollary \ref{RankOneCylConfig}.  Also by Corollary \ref{RankOneCylConfig}, for each $j$, with $2 \leq j \leq g$, it suffices to let each cycle $a_j$ be a curve lying entirely in a small neighborhood of the boundaries of two adjacent cylinders defined so that $a_i \cap a_j = \emptyset$, for $i \not= j$.  To see that such a choice is possible, cut each cylinder along its core curve thereby separating $(X,\omega)$ and collapse each core curve to a point to get a collection of positive genus surfaces $(M_i,\eta_i)$, $1 \leq i \leq r$.  Each surface $(M_i,\eta_i)$ admits a homology basis with $a$-cycles that are pairwise disjoint.  Since the two discs (formerly half-cylinders) forming $(M_i,\eta_i)$ are homologically trivial, the $a$-cycles must lie in a small neighborhood of the boundaries of the discs.  Hence, they lie in a small neighborhood of the boundaries of the cylinders on $(X,\omega)$.

The choice of $b$-cycles will not matter except for the cycle $b_1$, which, by necessity, must traverse the height of every cylinder once in order to intersect the cycle $a_1$ exactly once.  Define $\mathcal{B}(X,\omega) = \{a_j, b_j | 1 \leq j \leq g\}$.

\begin{figure}
 \centering
 \includegraphics[width=80mm]{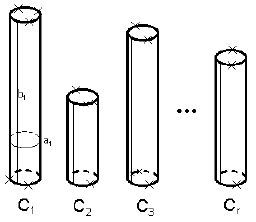}
 \caption{Homology Basis: Saddle connections on the top of $C_i$ are identified with saddle connections on the bottom of $C_{i+1}$.  Crosses denote copies of zeros.}
 \label{HomBasis}
\end{figure}

\section{Field of Definition}

\begin{lemma}
\label{RatFieldDef}
Let $\mathcal{M}$ be an AIS with completely degenerate KZ-spectrum.  Then $\textbf{k}(\mathcal{M}) = \mathbb{Q}$.
\end{lemma}

\begin{proof}
Let $(X,\omega) \in \mathcal{M}$.  Then $(X,\omega)$ is completely periodic by \cite[Thm. 5.5]{AulicinoCompDegKZ}.  Without loss of generality, assume the horizontal foliation of $(X,\omega)$ is periodic.  Furthermore, any periodic foliation can be written as a union of $r$ cylinders of unit circumference, where $r$ depends on the foliation, and the cylinders are arranged as in Corollary \ref{RankOneCylConfig}, see Figure \ref{HomBasis}.  We choose a specific closed loop in $\mathcal{M}$ that will simplify our calculations of the monodromy matrix below.

Consider the orbit of $(X,\omega)$ under the horocycle flow.  By \cite[Prop. 4(2)]{SmillieWeissMinSets}, the closure of this orbit is isomorphic to a $d$-dimensional torus, $\mathbb{T}^d$, for some $d$.  If we consider an $\varepsilon > 0$ ball $B$ about $(X,\omega)$ in $\mathbb{T}^d \subset \mathcal{M}$, then there is a sufficiently large value of $T >> 1$ such that $h_T \cdot (X,\omega) := (X_1, \omega_1) \in B$ and the following property holds.  Since we reach $(X_1, \omega_1)$ from $(X,\omega)$ by the horocycle flow, both $(X,\omega)$ and $(X_1,\omega_1)$ admit cylinder decompositions that differ only by the twisting of the cylinders themselves.  Therefore, it is possible to consider the continuous path $\gamma': [0,1] \rightarrow B$, from $(X_1, \omega_1)$ to $(X,\omega)$ lying entirely in $\mathbb{T}^d$ because $\mathbb{T}^d$ is a closed manifold.  We assumed $T$ to be sufficiently large so that following the path $\gamma'$ will not undo the twists induced by the horocycle flow.  Let $\gamma$ denote the path along the trajectory of the horocycle flow followed by concatenation with $\gamma'$.

Next we compute the monodromy matrix, denoted by $A$, of $[\gamma] \in \pi_1(\mathcal{M})$.  To do this, we pass freely between cohomology and homology by Poincar\'e duality.  Fix the homology basis $\mathcal{B}$ defined in Section \ref{CylConfigHomBasis} and compute the monodromy matrix induced by following $\gamma$.  For each of the $r$ cylinders, traveling along $\gamma$ results in $C_k$ twisting $\rho_k \geq 0$ times, with $\rho_k \in \mathbb{Z}$.  It is crucial to note that the assumption above on $T$ being sufficiently large guarantees that there exists a value of $k$ such that $\rho_k > 0$.  Traveling along $\gamma$ adds an $a_1$ cycle to every homology cycle traversing the height of every cylinder, while every other cycle is fixed.

By our choice of $a$-cycles, all of the $a$-cycles are fixed along the path $\gamma$, so that block of the monodromy matrix is the identity and the off-diagonal block that records the addition of $b$-cycles to $a$-cycles is the zero matrix.  Furthermore, the twisting does not combine any pair of $b$-cycles and so that block of the monodromy matrix is also the identity.  This follows because no linear combination of $b$-cycles is equal to $a_1$.  Therefore the monodromy matrix $A$ corresponding to the path $[\gamma]$ is given as follows: let the first $g$ rows and columns of $A$ correspond to the $a$-cycles and the second set of $g$ rows and columns of $A$ correspond to the $b$-cycles.  Then,
$$A = \left[ \begin{array}{cc}
\text{Id} & 0\\
M & \text{Id} \end{array} \right],$$
where each block of $A$ is a $g \times g$ matrix, and $M$ will be determined below.

We claim that $M \not= 0$, which will imply that $A$ is unipotent and $A \not= \text{Id}$.  It suffices to find a single entry in $M$ that is non-zero.  Recall that the cycle $b_1 \in \mathcal{B}$ was chosen so that it traversed the height of every cylinder.  Let $\rho = \sum_k \rho_k$ and recall that $\rho > 0$.  Therefore, following the path $\gamma$ will add $\rho$ copies of $a_1$ to $b_1$, which implies that the $1,1$ entry of $M$ is $\rho \not= 0$.

Since $A$ is unipotent, it generates a non-compact subgroup of matrices in $\text{SL}_{2g}(\mathbb{Z})$.  If we consider the Jordan form of $A$, then there is at least one Jordan block of dimension at least $2 \times 2$ that is unipotent because matrix conjugation preserves unipotence.  By \cite[Thm. 1.3]{AvilaEskinMollerForniBundle} and Lemma \ref{MaxForniSp}, if $\nu$ is a finite affine measure on $\mathcal{M}$, then for $\nu$-a.a. $x \in \mathcal{M}$, $H^1(X, \mathbb{R}) = p(T_{\mathbb{R}}(\mathcal{M}))(x) \oplus F(x)$.  By the definition of $F(x)$, every monodromy matrix of $\mathcal{M}$ must decompose into blocks such that there is a block of dimension $2g-2 \times 2g-2$ lying in $\text{SO}(2g-2)$.  This implies that $A$ is conjugate to a matrix whose Jordan decomposition consists of a $2 \times 2$ unipotent matrix and a $2g-2 \times 2g-2$ diagonal matrix with unit eigenvalues.

By \cite[Thm. 1.4]{WrightFieldofDef}, we have a decomposition of $H^1$ into subbundles $(\oplus_{\rho} \mathbb{V}_{\rho}) \oplus \mathbb{W}$, where the direct sum of $\mathbb{V}_{\rho}$ is taken over all Galois conjugates $\rho$ of $\textbf{k}(\mathcal{M})$.  This induces a decomposition of $A$ into blocks corresponding to each of these subbundles.  By \cite[Thm. 1.4]{WrightFieldofDef}, $\mathbb{V}_{\text{Id}} = p(T_{\mathbb{R}}(\mathcal{M}))$.  Also, by \cite[Thm. 1.4]{WrightFieldofDef} and the fact that $\text{SO}(n)$ does not contain any unipotent elements, the block corresponding to the subbundle $\mathbb{V}_{\text{Id}}$ must be the $2 \times 2$ unipotent block.  Furthermore, Galois conjugates of a unipotent block are unipotent.  However, any non-trivial Galois conjugate of $\mathbb{V}_{\text{Id}}$ would have to send the $2 \times 2$ unipotent block into the block contained in the compact group $\text{SO}(2g-2)$.  Such a contradiction implies that there can be no nontrivial Galois conjugates of $\textbf{k}(\mathcal{M})$.  Thus, $\mathbf{k}(\mathcal{M}) = \mathbb{Q}$.
\end{proof}

\begin{corollary}
\label{DenseSubset}
If $\mathcal{M}$ is an AIS with completely degenerate KZ-spectrum, then it contains a dense subset of arithmetic Teichm\"uller curves.
\end{corollary}

\begin{proof}
By Lemma \ref{RatFieldDef}, $\textbf{k}(\mathcal{M}) = \mathbb{Q}$.  By the definition of $\textbf{k}(\mathcal{M})$, $\mathcal{M}$ can be written locally as the zero set of a set of linear equations with coefficients in $\mathbb{Q}$.  Hence, it admits a dense set of rational solutions.  Each rational solution corresponds to a square-tiled surface because the periods lie in $(1/q)\mathbb{Z}$, where $q \in \mathbb{Z}^+$ is the lowest common denominator of the rational solution.  Finally, the arithmetic Teichm\"uller curves must be dense because the square-tiled surfaces are dense.
\end{proof}

\section{The Map $\pi$}

In this section we use Lemma \ref{RatFieldDef} to prove that if $\cM$ is an affine manifold with completely degenerate KZ-spectrum, then its elements can be realized as torus covers.  Loosely speaking, this allows us to consider the degenerations that result from colliding branch points on the torus while fixing the torus itself.  The resulting degenerations will be very specific, and considering all of the combinatorial possibilities will allow us to prove the main theorem.

\subsection{$\cM$ as a Family of Torus Covers}

We begin with a technical lemma.  Let $\mathcal{M}$ be an AIS.  Given $(X,\omega) \in \mathcal{M}$, define the set
$$L(X,\omega) := \left\{\int_{\gamma_i}\omega \big\vert \gamma_i \in H_1(X,\mathbb{Z})\right\}.$$

\begin{lemma}
\label{LDiscreteGen}
Let $\mathcal{M}$ be an AIS with $\dim p(T_{\mathbb{R}}(\mathcal{M})) = 2$ and $\textbf{k}(\mathcal{M}) = \mathbb{Q}$.  If $(X,\omega) \in \mathcal{M}$, then $L(X,\omega)$ is a lattice in $\mathbb{C}$.
\end{lemma}

\begin{proof}
Since $\mathcal{M}$ is an AIS, locally identify $\mathcal{M}$ with its tangent space $T_{\mathbb{R}}(\mathcal{M})$.  Let $T_{\mathbb{R}}(\mathcal{M})$ have codimension $d$.  Since $T_{\mathbb{R}}(\mathcal{M})$ is a linear subspace in period coordinates, there is a matrix $A \in \text{Mat}_{d \times k}(\textbf{k}(\mathcal{M}))$ such that locally $\ker(A) = T_{\mathbb{R}}(\mathcal{M})$.  Note that the natural map $p: H^1(X, \Sigma, \mathbb{R}) \rightarrow H^1(X, \mathbb{R})$ is a linear transformation of vector spaces.  Hence, $p(\ker(A))$ is in fact a subspace of $H^1(X, \mathbb{R})$.  Therefore, there exists another matrix $A'$ such that $\ker(A') = p(T_{\mathbb{R}}(\mathcal{M}))$, and by assumption, $A' \in \text{Mat}_{2g-2 \times 2g}(\mathbb{Q})$.  Without loss of generality, let $A' := (a_{ij})$ be in reduced row echelon form.

Fix a basis $\{\gamma_1, \ldots, \gamma_k\}$ for $H_1(X,\Sigma, \mathbb{Z})$ so that $\{\gamma_1, \ldots, \gamma_{2g}\}$ forms a basis for the absolute homology $H_1(X,\mathbb{Z})$.  This induces a basis $\{\gamma_1^*, \ldots, \gamma_k^*\}$ on cohomology by Poincar\'e duality so that $\{\gamma_1^*, \ldots, \gamma_{2g}^*\}$ forms a basis for $H^1(X,\mathbb{Z})$.  If we consider period coordinates with respect to this basis, we get
$$\Phi(X,\omega) = \left(\int_{\gamma_1} \omega, \ldots, \int_{\gamma_k}\omega\right) \in \mathbb{C}^k.$$
Under the map $p$, we get
$$p\left(\int_{\gamma_1} \omega, \ldots, \int_{\gamma_k}\omega\right) = \left(\int_{\gamma_1} \omega, \ldots, \int_{\gamma_{2g}}\omega\right) := y.$$
Then $y$ satisfies $A'y = 0$.  Hence, for each $1 \leq i \leq 2g$,
$$\int_{\gamma_i} \omega = a_{i, 1} \int_{\gamma_1}\omega  + a_{i, 2} \int_{\gamma_2}\omega,$$
where $a_{ij} \in \mathbb{Q}$.  Let $p$ be a positive integer such that $p a_{ij} \in \mathbb{Z}$, for all $i,j$.  This implies that every element in $pL$ is an integral linear combination of two complex numbers, thus $pL$ is a lattice, as is $L$.
\end{proof}

By Lemma \ref{MaxForniSp}, if $\mathcal{M}$ is an AIS with completely degenerate KZ-spectrum, then $\dim p(T_{\mathbb{R}}(\mathcal{M})) = 2$, and Lemma \ref{RatFieldDef} implies $\textbf{k}(\mathcal{M}) = \mathbb{Q}$.  This yields 

\begin{corollary}
\label{LDiscrete}
Let $\mathcal{M}$ be an AIS with completely degenerate KZ-spectrum.  If $(X,\omega) \in \mathcal{M}$, then $L(X,\omega)$ is a lattice in $\mathbb{C}$.
\end{corollary}

Let $(X,\omega) \in \mathcal{M}$, where $\mathcal{M}$ is an AIS with completely degenerate KZ-spectrum, and thus by Lemma \ref{RatFieldDef}, $\textbf{k}(\mathcal{M}) = \mathbb{Q}$.  Let $z_0 \in X$.  Define the function
$$\pi(z) := \int_{z_0}^z \omega \mod L.$$
This function is well-defined because integrals along paths on $X$ are defined up to integrals along absolute homology classes.  Therefore, after modding out by this discrepancy, i.e. $L$, which in our case forms a lattice in $\mathbb{C}$ by Corollary \ref{LDiscrete}, we get a map from $X$ to the torus $\mathbb{T}^2$.  We show that restricting to the subset of $\mathcal{M}$ consisting of square-tiled surfaces, the map $\pi$ agrees with the map $\pi_{opt}$ up to the finite set of branch points, which are forgotten under $\pi$.

\begin{lemma}
\label{PiSqTiled}
Let $\mathcal{M}$ be an AIS with completely degenerate KZ-spectrum.  If $(X,\omega) \in \mathcal{M}$ is a square-tiled surface, then $\pi = \pi_{opt}$ up to isogeny, i.e. $\pi$ is a covering map of the torus with no non-trivial intermediate toral covers.
\end{lemma}

\begin{proof}
First we claim that $\pi$ is a covering map of the torus after puncturing the zeros of $\omega$ and their images under $\pi$.  Consider an open set $U \subset E^*$, where $E^*$ denotes the punctured torus.  Then the only sets which map to $U$ under $\pi$ are exactly the translates of $U$ by $L$.  For all $P \in E^*$, $|\pi^{-1}(P)|$ is a fixed constant $d$.  Hence, $\pi$ is a covering map of degree $d$.

By definition of $\pi_{opt}$, any covering map $\pi_0$ can be factored as $\pi_0 = \iota \circ \pi_{opt}$, where $\iota$ is an isogeny and $\pi_{opt}$ is a covering map of the torus with no intermediate torus covers.  Hence, $\pi = \iota \circ \pi_{opt}$.  Finally, the lattice $L$ can be taken so that there is no intermediate torus cover by letting $L$ be the ``coarsest'' lattice containing $L(X,\omega)$, i.e. for all $L' \supset L(X,\omega)$, $L \subset L'$.
\end{proof}


The following lemma is essentially \cite[Lemma 2.1]{EskinMasurSchmollRectBar}, which is proven in the stratum $\mathcal{H}(1,1)$.  However, the argument is completely general and written in full generality in the arXiv version of their paper.

\begin{lemma}
\label{PiHasSameDegree}
Let $\{\pi^{(n)}_{opt}\}_{n=1}^{\infty}$ be a sequence of optimal branched coverings of a torus, each having degree $d_{opt}$.  For any subsequence of $\{\pi^{(n)}_{opt}\}_{n=1}^{\infty}$ converging to a torus covering $\pi'$, $\pi'$ has the same degree as each of the maps $\pi^{(n)}_{opt}$.
\end{lemma}

\begin{proof}
Let $\pi^{(n)}_{opt}: (X_n, \omega_n) \rightarrow (\mathbb{T}_n^2, dz)$.  Fixing the area of the torus $(\mathbb{T}_n^2, dz)$ to be one, for all $n$, implies the area of $(X_n,\omega_n)$ is $d_{opt}$.  By assumption, $\{\pi^{(n)}_{opt}\}_{n=1}^{\infty}$ is a sequence of holomorphic maps converging to a holomorphic map.  This implies that $(X_n,\omega_n)$ is a sequence of surfaces of fixed area $d_{opt}$ converging to a surface also carrying a holomorphic differential, and therefore, the area must be preserved along this sequence.  This implies that the limit of these surfaces $(X', \omega')$ must also have area $d_{opt}$ and $\pi'_{opt}$ must be a covering map of $(X', \omega')$ to the torus, whose area is also preserved.  This forces the degree of $\pi'_{opt}$ to be $d_{opt}$, the ratio of the areas.
\end{proof}

\subsection{Review of Results for Teichm\"uller Curves}

The results in this subsection summarize the results of \cite{MollerShimuraTeich} that will be used in this paper.  The following lemma is proven in \cite{MollerShimuraTeich}, and we present an alternate proof here.

\begin{lemma}
\label{TeichCurveIsArith}
A Teichm\"uller curve with completely degenerate KZ-spectrum is arithmetic, i.e. it is generated by a square-tiled surface.
\end{lemma}

\begin{proof}
By Lemma \ref{RatFieldDef}, any Teichm\"uller curve with completely degenerate KZ-spectrum has rational field of definition.  By definition, Teichm\"uller curves defined over $\mathbb{Q}$ are arithmetic and are generated by square-tiled surfaces by \cite{GutkinJudge}.
\end{proof}

Let $E_m = \mathbb{C}/(m\mathbb{Z}[i])$ be a torus.  The lattice points of $E_m$ are called \emph{$m$-torsion points}.  Define $E = E_1$. By scaling, it makes sense to discuss the $m$-torsion points of $E$.  Any Veech surface $(X,\omega)$ with completely degenerate KZ-spectrum is square-tiled by Lemma \ref{TeichCurveIsArith}.  Hence, it admits an optimal covering $\pi_{opt}: X \rightarrow E$, which is branched over finitely many points $B \subset E$, and optimal in the sense that there is no intermediate torus covering.  Let $E^* = E \setminus B$.  The degree of $\pi_{opt}$ is denoted by $d_{opt}$.  We state an abbreviated version of \cite[Cor. 5.15]{MollerShimuraTeich}.

\begin{corollary}
\label{MoellerCor515}
If $(X, \omega)$ is a Veech surface generating a Teichm\"uller curve with completely degenerate KZ-spectrum, then $X$ has genus $g = 3,4,5$, $(X,\omega)$ admits a degree $d_{opt}$ ramified covering of the torus, where $d_{opt}$ can be explicitly calculated, and depends only on the stratum in which $(X, \omega)$ lies.  Furthermore, $d_{opt} \leq 36$, and $d_{opt}$ is different for every stratum.
\end{corollary}

A cylinder $C$ in $E^*$ is bounded by points in $B$ and lifts to one or more cylinders in $\pi_{opt}^{-1}(C) \subset X$.  The following lemma is \cite[Lem. 5.17]{MollerShimuraTeich} and it fully describes how cylinders bounded by punctures on the torus $E^*$ lift to the Veech surface $(X,\omega)$, where $X$ has genus five.

\begin{lemma}[M\"oller]
\label{MollerCaseLemma}
In each of the possible strata in genus five containing Teichm\"uller curves with completely degenerate KZ-spectrum, one of the following three possibilities holds.
\begin{itemize}
\item[(i)] The preimage under $\pi_{opt}$ of each cylinder in $E^*$ consists of only one cylinder in $X$, or
\item[(ii)] $B$ consists of one element only and each cylinder in $E^*$ has $k$ preimages under $\pi_{opt}$ where $2 \leq k \leq 4$, or
\item[(iii)] $B$ is contained in the set of $2$-torsion points of $E^*$. Moreover, each cylinder in $E^*$ has two preimages under $\pi_{opt}$.
\end{itemize}
\end{lemma}

We conclude with a simple corollary of these results that will be useful later.

\begin{lemma}
\label{CaseiReduction}
Any infinite family of Teichm\"uller curves $\{\mathcal{C}_n'\}_{n=1}^{\infty}$ in genus five with completely degenerate KZ-spectrum contains an infinite family of curves $\{\mathcal{C}_n\}_{n=1}^{\infty}$ such that for each $n$, there is a Veech surface $(X_n, \omega_n) \in \mathcal{C}_n$ such that $(X_n, \omega_n)$ satisfies Case i) of Lemma \ref{MollerCaseLemma}.
\end{lemma}

\begin{proof}
We show that there are at most finitely many examples of Teichm\"uller curves in genus five, not obeying the assumptions of this lemma.  Consider Cases ii) and iii) of Lemma \ref{MollerCaseLemma}.  In each of these cases, the branch points on the torus lie at either the origin or the 2-torsion points.  Hence, the torus can be realized as a single square or four unit squares, respectively.  This implies that the surface covering the torus can be realized by at most $4d_{opt}$ squares.  Hence, there can be only finitely many such examples in these cases.
\end{proof}

Since a Teichm\"uller curve contained in an AIS with completely degenerate KZ-spectrum must also have completely degenerate KZ-spectrum, we combine Corollary \ref{LDiscrete} and Lemmas \ref{PiSqTiled}, \ref{PiHasSameDegree}, \ref{MollerCaseLemma}, and \ref{CaseiReduction} to prove

\begin{corollary}
\label{CaseiRedCor}
Let $\mathcal{M}$ be an AIS in genus five with completely degenerate KZ-spectrum that is not a Teichm\"uller curve.  If $(X,\omega) \in \cM$, then $M$ is a torus cover of degree $d_{opt}$ depending only on the stratum in which $\cM$ lies, and the covering of $(X,\omega)$ to the punctured torus satisfies the Case i) property of Lemma \ref{MollerCaseLemma}.
\end{corollary}

\begin{proof}
It suffices to show that the limit of torus covers satisfying Case i) also satisfies Case i).  Furthermore, it suffices to assume that all limits are contained in $\cM \subset \mathcal{H}(\kappa)$, i.e. no zeros collide.  Let $\cM_0 \subset \cM$ be the union of Teichm\"uller curves not satisfying Case i).  By the proof of Lemma \ref{CaseiReduction}, $\cM_0$ is a finite union of Teichm\"uller curves.  Let $\cC$ be a Teichm\"uller curve in $\cM_0$.  Then there is a sequence of torus covers $(X_n, \omega_n) \in \cM \setminus \cM_0$ converging to a torus cover $(X', \omega')$ in $\cC$.  Let $C_n \subset E^{(n)*}$ be a sequence of cylinders that converges to a cylinder $C' \subset E'^*$.  For all $n$, $C_n$ lifts to a unique cylinder.  However, $C'$ must also lift to a unique cylinder corresponding to the limit of the lifts of $C_n$ under $\pi_{opt}$.  Moreover, there cannot be an intermediate torus cover where $C'$ lifts to more than one cylinder because all lifts of $C'$ lift to one cylinder.  Hence, $\cM_0$ is empty.
\end{proof}

\subsection{Partial Compactification of Moduli Space}

In recent work of Mirzakhani and Wright \cite{MirzakhaniWrightBoundary}, they developed a partial compactification of moduli space well-suited to the study of affine invariant submanifolds.  Roughly speaking the idea is to consider limits of translation surfaces by taking limits in period coordinates, where all of the period coordinates remain finite and any period coordinates that converge to zero are discarded.  The theorem proven by them describes the tangent space of the boundary affine manifold when the boundary translation surface is connected.  It is conjectured that even if the boundary component is a ``multi component translation surface,'' then the boundary also contains an affine manifold.  However, it relies on a generalization of \cite{EskinMirzakhaniMohammadiOrbitClosures} and it can be avoided here due to an observation, cf. Lemma \ref{BdMIsConn}.

We outline our strategy here.  We use the definition of convergence in the context of \cite{MirzakhaniWrightBoundary} and assume that there exists a convergent sequence of degenerating translation surfaces.  If $(X', \omega')$ denotes the boundary translation surface, we use the maximality of the Forni subspace to argue that $(X', \omega')$ is connected in the sense of \cite{MirzakhaniWrightBoundary}.  This allows us to circumvent the ''multi component'' conjecture and apply the theorem of \cite{MirzakhaniWrightBoundary}.  Next we show that the Forni subspace of any boundary affine manifold must also be maximal.  Finally, we establish that there always exists a boundary affine manifold, and this allows the degeneration arguments of the next section to begin.

\begin{definition}
A \emph{multicomponent translation surface} is a collection $(X, \omega, \Sigma)$. Here $X$ is a compact Riemann surface with at most finitely many connected components, $\omega$ is an Abelian differential that is nonzero on every connected component of $X$, and $\Sigma \subset X$ is a finite set of marked points. We require that $\Sigma$ contain all of the zeros of $\omega$.  If there are $k$ connected components, we will assume that they are labeled by $\{1, \ldots, k\}$.
\end{definition}

We will use the terms \emph{marked points} and \emph{punctures} interchangeably throughout this paper.

\begin{definition}
Let $\cH$ and $\cH'$ be two strata of multicomponent translation
surfaces.  We say that $(X_n, \omega_n, \Sigma_n) \in \cH$ converges to $(X', \omega', \Sigma') \in \cH'$ if there are decreasing neighborhoods $U_n \subset X'$ of $\Sigma'$ with $\cap U_n = \Sigma'$ such that the following holds. There are maps $g_n: X_n \setminus U_n \rightarrow X' \setminus U_n$ that are diffeomorphisms onto their range, such that
\begin{itemize}
\item (1) $g_n^*(\omega_n)$ converges to $\omega'$ in the compact open topology on $X' \setminus \Sigma'$,
\item (2) the injectivity radius of points not in the image of $g_n$ goes to zero uniformly in $n$.
\end{itemize}
Injectivity radius at a point is defined as the sup of $\varepsilon$ such that the ball of radius $\varepsilon$ in the flat metric centered at that point is embedded and does not contain any marked points.
\end{definition}

\begin{lemma}
\label{BdMIsConn}
Let $\cM$ be an AIS with completely degenerate KZ-spectrum.  If $\{(X_n, \omega_n)\}$ is a sequence in $\cM$ converging to $(X', \omega')$ in the sense defined above, then $(X', \omega')$ is connected, i.e. in the definition of a multi-component translation surface, $k=1$.
\end{lemma}

\begin{proof}
By contradiction, assume that $(X', \omega')$ is not connected.  Then there exist at least two connected components $(X'_1, \omega'_1)$ and $(X'_2, \omega'_2)$.  By the definition of convergence, we have $\omega'_i|_{X_i'} \not\equiv 0$ and $\omega'_i|_{X_j'} \equiv 0$, for $i \not= j$.  Moreover, by definition, $\omega'|_{X'_1} = \omega'_1$ and $\omega'|_{X'_2} = \omega'_2$.  It is clear that as elements of the finite dimensional space of Abelian differentials on $(X', \omega')$, $\omega'_1$ and $\omega'_2$ are linearly independent.  For each $i, j \in \{1,2\}$, we have
$$\int_{X'}\omega_i'\omega_i'\frac{\omega'}{\omega'} = \int_{X'_i}|\omega_i'|^2 \not= 0,$$
and $i \not= j$
$$\int_{X'}\omega_i'\omega_j'\frac{\omega'}{\omega'} = 0.$$
Hence, at the boundary, there is a $2 \times 2$ minor of the derivative of the period matrix of full rank.  By either \cite{FayThetaFcns, Yamada} or \cite[Sect. 3]{AthreyaBufetovEskinMirzakhani}, there exists $N$ sufficiently large such that the derivative of the period matrix $(X_n, \omega_n) \in \cM$ has a $2 \times 2$ minor of full rank as well.  This contradicts the fact that $\cM$ has maximal Forni subspace and shows that every boundary translation surface is connected.
\end{proof}

\begin{proposition}\cite{MirzakhaniWrightBoundary}
\label{MWCollapseMaps}
Suppose that $(X_n, \omega_n, \Sigma_n) \in \cH$ converges to $(X', \omega', \Sigma') \in \cH'$.  Then there are collapse maps $f_n: X_n \rightarrow X'$ that map marked points to marked ponits and have the following property: for any $\varepsilon > 0$ and $n$ sufficiently large, the $f_n$ are the inverse to $g_n$ on the subset $(X_n, \omega_n, \Sigma_n)$ with injectivity radius at least $\varepsilon$.
\end{proposition}

This allows us to define the space of vanishing cycles:
$$V_n := \text{ker}(f_n: H_1(X_n, \Sigma_n, \bC) \rightarrow H_1(X', \Sigma', \bC)).$$
The significance of this sequence of spaces is captured by \cite[Prop. 2.4]{MirzakhaniWrightBoundary}.

\begin{proposition}\cite{MirzakhaniWrightBoundary}
\label{VanEventConst}
Each space $V_n$ is well-defined, and the sequence of spaces $\{V_n\}$ is eventually constant.
\end{proposition}

In light of this proposition, it suffices to assume $n$ is sufficiently large and suppress the index from the notation.  The following is \cite[Prop. 2.5]{MirzakhaniWrightBoundary}.

\begin{proposition}\cite{MirzakhaniWrightBoundary}
\label{BdTangIsAnn}
In the situation above, $\text{Ann}(V)$ is naturally identified with the tangent space to the boundary stratum $\cH'$.  Furthermore, there is a neighborhood of $0$ in $\text{Ann}(V)$ such that if $\xi_n$ and $\xi'$ are in this neighborhood and $\xi_n$ converges to $\xi'$, then $(X_n, \omega_n, \Sigma_n) + \xi_n$ converges to $(X', \omega', \Sigma') + \xi' \in \cH'$.
\end{proposition}

Finally, we state \cite[Thm. 2.7]{MirzakhaniWrightBoundary}.

\begin{theorem}[\cite{MirzakhaniWrightBoundary}]
\label{MWBdTanSpThm}
Let $\cM$ be an affine invariant sub-manifold.  The subset of the boundary of $\cM$ consisting of connected surfaces consists of a finite union of affine invariant submanifolds.  If $(X_n, \omega_n) \in \cM$ converge to a connected translation surface $(X', \omega')$ in the boundary, then $(X', \omega')$ is contained in a component $\cM'$ of the boundary of $\cM$ whose tangent space can be computed, for infinitely many $n$, as
$$T(\cM') = T_{(X_n, \omega_n)}(\cM) \cap \text{Ann}(V).$$
Here $V = V_n$ is the space of vanishing cycles, which is eventually constant.
\end{theorem}

The addition of a prime on any object will indicate that it is a boundary object.

\begin{lemma}
\label{BdFSubspIsMax}
Let $\cM$ be an AIS with completely degenerate KZ-spectrum.  If $\cM'$ is a boundary affine manifold, then $\cM'$ has completely degenerate KZ-spectrum.
\end{lemma}

\begin{proof}
By Lemma \ref{BdMIsConn}, the affine manifold $\cM'$ is an affine manifold of connected translation surfaces.  Let $\cM' \subset \cH'$.  By Proposition \ref{BdTangIsAnn}, the tangent space to the boundary stratum $\cH'$ is identified with $\text{Ann}(V)$.  Therefore, we have an isomorphism of $p(T_{\bC}(\cH'))$ with $p(\text{Ann}(V))$.  Since there is also an isomorphism of $p(T_{\bC}(\cM'))$ with $p(T_{\bC}(\cM))) \subseteq p(\text{Ann}(V))$, there is an isomorphism of the symplectic, Hodge, or orthogonal complement of $p(T_{\bC}(\cM'))$ with that of $p(T_{\bC}(\cM) \cap \text{Ann}(V))$, which is contained in the Forni subspace of $\cM$ by \cite{AvilaEskinMollerForniBundle}.  Hence, the entire complement of $p(T_{\bC}(\cM'))$ in $p(\text{Ann}(V))$ is a Forni subspace.
\end{proof}

\begin{corollary}
\label{DegSurfsCor}
Let $\cM$ be an AIS with completely degenerate KZ-spectrum.  If $\cM'$ is an affine manifold in the boundary of $\cM$, then either $\cM'$ lies in the moduli space of genus five translation surfaces, $\cM'$ is the Teichm\"uller curve generated by the Ornithorynque with marked points, the EW with marked points, or the torus with marked points.
\end{corollary}

\begin{proof}
By Lemma \ref{BdFSubspIsMax}, $\cM'$ has completely degenerate KZ-spectrum.  By Corollary \ref{DenseSubset}, an AIS with completely degenerate KZ-spectrum must contain infinitely many Teichm\"uller curves.  Those Teichm\"uller curves must also have completely degenerate KZ-spectrum because if the KZ-cocycle acts isometrically over an \splin -invariant set $\mathcal{M}$, then it acts isometrically on any \splin -invariant subset of $\mathcal{M}$.  Since there are exactly two such Teichm\"uller curves except possibly in genus five, where there may be infinitely many, the corollary follows.\footnote{By convention, we consider the torus to always have completely degenerate KZ-spectrum because it always has a non-zero $1 \times 1$ derivative of its period matrix.}
\end{proof}

\subsection{Existence of a Degenerate Surface}

Up until this point, all of our results concerning the boundary of an AIS with completely degenerate KZ-spectrum have assumed the existence of a degenerate translation surface with either fewer zeros or lower genus.  We now prove that every such AIS does indeed contain such a degenerate translation surface.  Let
$$a(t) = \left[ \begin{array}{cc}
1 & 0\\
0 & e^{t} \end{array} \right].$$

\begin{definition}
Let $\{C_1, \ldots, C_n\}$ be a collection of horizontal cylinders on a translation surface $(X, \omega)$.  For each cylinder $C_i$, act on it by $a(t_i)$.  The collection of matrices, which we denote $a(t_1, \ldots, t_n)$, acting on $\{C_1, \ldots, C_n\}$, is called a \emph{REL stretch}, if the absolute periods of $(X, \omega)$ are fixed under this deformation.  In other words, each cylinder may be acted upon by a different matrix $a(t_i)$, but all of the absolute periods on the surface are fixed.
\end{definition}

\begin{lemma}
\label{DegSurfExists}
Let $\cM$ be an AIS with completely degenerate KZ-spectrum that is not a Teichm\"uller curve.  Then there exists $(X',\omega') \in \cM'$ and a sequence $\{(X_n, \omega_n)\}$ converging to $(X', \omega')$ such that $X'$ either has strictly lower genus than $X$, or $\omega'$ has strictly fewer zeros than $\omega$.
\end{lemma}

\begin{proof}
By Lemma \ref{RankOneCylConfig}, every periodic direction on every translation surface in $\cM$ is a collection of homologous cylinders.  By the assumption that $\cM$ is not a Teichm\"uller curve, and the fact that the dimension of the absolute tangent space is minimal, we have that \cite[Thm. 3.3]{AulicinoZeroExpGen3} implies that there exists a horizontally periodic translation surface $(X, \omega) \in \cM$ and a REL stretch that fixes the total heights of the cylinders, but varies the heights of some collection of cylinders.  This REL stretch can be continued until cylinders collapse, i.e. their height converges to zero.  Since every cylinder is homologous, the total heights of the cylinders represent an absolute period, which is non-zero and fixed under this deformation by definition.  Let $\{(Y_n, \theta_n)\}$ be a sequence in $\cM$ in which one or more cylinders collapse, and let
$$h(t) = \left[ \begin{array}{cc}
1 & t\\
0 & 1 \end{array} \right].$$
Finally, let $\{(X_n, \omega_n) := h(t_n) \cdot (Y_n, \theta_n)\}$ be a sequence such that each element admits at least one vertical saddle connection.  This means that at least one vertical saddle connection is degenerating, so the limit is a degenerate surface with lower genus or fewer zeros.
\end{proof}

\begin{remark}
The key observation for the degeneration arguments in this paper involves reconciling the two ways of viewing a translation surface in an AIS $\cM$ with completely degenerate KZ-spectrum.  On the one hand, it can be viewed as a collection of homologous cylinders.  However, it is far more valuable to observe that by Corollary \ref{LDiscrete} every translation surface in $\cM$ admits a covering to a torus branched over finitely many points.  Since it is possible to deform this surface while fixing absolute periods, it implies that we can fix a specific torus and deform the surface by letting the branch points (marked points) move.  Of course, we cannot choose how those branch points move, but it will suffice to assume that such a deformation exists.
\end{remark}

\section{Proof of the Main Theorem}

Recall that there is exactly one Teichm\"uller curve with completely degenerate KZ-spectrum in each of the two moduli spaces of Abelian differentials over surfaces of genus three and four by \cite{MollerShimuraTeich}.  Furthermore, these are the only AIS with completely degenerate KZ-spectrum in those genera by \cite{AulicinoCompDegKZ}.  By the proof of Corollary \ref{DegSurfsCor}, we have an alternate proof of this result from \cite{AulicinoCompDegKZ}, which we record as an independent proposition here.

\begin{proposition}
\label{MainThmExcGen5}
There do not exist affine manifolds with completely degenerate KZ-spectrum that are not Teichm\"uller curves, except possibly in the moduli space of Abelian differentials on genus five surfaces.
\end{proposition}

The final claim needed to prove Theorem \ref{theorem:MainThm} is that the surface $(X', \omega')$ whose existence is guaranteed by Lemma \ref{DegSurfExists} has genus five.  This implies that we can collapse zeros without pinching any curves, and this will yield the desired contradiction needed in proof of Theorem \ref{theorem:MainThm}.  We summarize this in Proposition \ref{ZerosDegSameGen} and split the proof into multiple lemmas for the ease of the reader.

\begin{proposition}
\label{ZerosDegSameGen}
If $\cM \subset \cH(\kappa)$ is an affine manifold with completely degenerate KZ-spectrum in genus five, then there exists a surface $(X',\omega') \in \partial \cM$ such that $(X', \omega')$ lies in a stratum of genus five surfaces in the boundary of $\cH(\kappa)$.
\end{proposition}

\begin{theorem}
\label{theorem:MainThm}
If $\mathcal{M}$ is an AIS with completely degenerate KZ-spectrum, then $\mathcal{M}$ is an arithmetic Teichm\"uller curve.
\end{theorem}

\begin{proof}
By contradiction, assume that $\mathcal{M}$ is not a Teichm\"uller curve.  By Proposition \ref{ZerosDegSameGen}, there exists a genus five surface $(X', \omega') \in \partial \cM$ that does not lie in the same stratum as $\cM$.  By \cite[Cor. 5.15]{MollerShimuraTeich}, the degree of each covering map $\pi_{opt}$ from a square-tiled surface in $\mathcal{M}$ to the torus is a fixed number $d_{opt}$.  By Lemma \ref{PiHasSameDegree}, the covering map of the torus by a surface generating $\mathcal{C}$ must also have degree $d_{opt}$.  Hence, $\mathcal{M}$ is an AIS in a stratum whose degree $d_{opt}$ is equal to the degree of the covering in an adjacent stratum.  However, by inspection of the table in \cite[Cor. 5.15]{MollerShimuraTeich}, no two strata have the same $d_{opt}$.  This contradiction implies that $\mathcal{M}$ is a Teichm\"uller curve, and it is arithmetic by Lemma \ref{TeichCurveIsArith}.
\end{proof}

\begin{corollary}
\label{corollary:MainCor}
There are at most finitely many Teichm\"uller curves with completely degenerate KZ-spectrum.
\end{corollary}

\begin{proof}
By contradiction, if there are infinitely many Teichm\"uller curves contained in a stratum, then they equidistribute to a higher dimensional affine manifold $\cM$ by \cite{EskinMirzakhaniMohammadiOrbitClosures}.  If each of the infinitely many Teichm\"uller curves has completely degenerate KZ-spectrum, then $\cM$ must as well because each of the Teichm\"uller curves are contained in the closed set \RankOne .  However, Theorem \ref{theorem:MainThm} proved that there are no such higher dimensional affine manifolds, so there cannot be infinitely many Teichm\"uller curves.
\end{proof}

\subsection{Proof of Proposition \ref{ZerosDegSameGen}}

We follow the notation of Proposition \ref{ZerosDegSameGen}.  By Corollary \ref{DegSurfsCor}, if $X'$ has lower genus than $X$, then after filling punctures, it is contained in the Teichm\"uller curve of either the Ornithorynque, Eierlegende Wollmilchsau (EW), or torus.

If $X'$ has lower genus than $X$, then there must be one or more curves on $X$ that collapse to one or more nodes of $\tilde X'$.  After removing all of the nodes of $\tilde X'$, let $p_i$ be a puncture on $X'$ and $\gamma_i$ a small curve on $X'$ homotopic to $p_i$.  We call the region bounded by $\gamma_i$ containing $p_i$ its interior.  Naturally, $\gamma_i$ can be pulled back to a curve on $\tilde X'$.  Then we say that two punctures $p_i$ and $p_j$ are \emph{connected} if there is a path from the interior of $\gamma_i$ to the interior of $\gamma_j$ that does not cross $\gamma_i$ for all $i$.  Observe that this implicitly allows the path to freely traverse the parts of $\tilde X'$ carrying the zero differential.  A maximal set of connected punctures is called a \emph{connected set of punctures} $\{p_1, \ldots, p_n\}$.

Given a translation surface $(X,\omega)$, and a point $p \in X$, we can consider the set $\Gamma_p$ of all closed trajectories through $p$.  The trajectories may pass through the zeros of $\omega$.  Define the set
$$C_{p}(X,\omega) = \bigcap_{\gamma \in \Gamma_p} \gamma.$$
Some elementary observations are in order.  First, for a torus $E$ it is clear that this set $C_{p}(E) = \{p\}$ by embedding the torus in the plane as a square and considering the horizontal and vertical trajectories in the torus intersecting only at $p$.  Secondly, for a torus cover, $C_{p}(X,\omega) \subseteq \pi_{opt}^{-1}(p)$, which can be seen by lifting the horizontal and vertical trajectories from the torus to $(X, \omega)$.  Finally, the concept of a connected set of punctures and the set $C_{p}(X,\omega)$ are naturally related by

\begin{lemma}
\label{ConnectedPctsAtIntersect}
If $\mathcal{P} = \{p_1, \ldots, p_n\}$ is a connected set of punctures on a degenerate surface $(X', \omega') \in$ \RankOne ~where $\omega'$ is holomorphic, then
$$\mathcal{P} \subseteq C_{p_1}(X', \omega').$$
\end{lemma}

\begin{proof}
Recall Corollary \ref{RankOneCylConfig}, which states that $(X, \omega)$ is a union of homologous cylinders.  In particular, it is true that cutting the core curves of two of the cylinders will separate the surface into two regions.  By contradiction, assume that there are two connected punctures $p_1, p_2$ that do not lie on the same curve.  Then there is a pair of parallel (homologous) flat trajectories $\eta_1$, $\eta_2$ dividing the surface into two regions such that $p_1$ lies on one region and $p_2$ lies on the other.  The curves $\eta_1$, $\eta_2$ determine cylinders.  Let $(X, \omega)$ be a non-degenerate surface in a small neighborhood of $(X', \omega')$.  Then $\eta_1$ and $\eta_2$ persist on $(X, \omega)$ and determine cylinders as well.

Let $\gamma_i$ be a small curve homotopic to $p_i$ on $X'$, for all $i$.  We call the interior of $\gamma_i$ the region bounded by $\gamma_i$ that is homeomorphic to a branched cover of the punctured disc.  Furthermore, $\gamma_i$ persists as a small closed curve on $X$ for each $i$.  However, on $(X, \omega)$ there is a saddle connection in the interior of $\gamma_1$ that is identified to a saddle connection in the interior of $\gamma_2$.  This contradicts Corollary \ref{RankOneCylConfig} because $\eta_1$ and $\eta_2$ are not homologous on $(X, \omega)$.
\end{proof}

\subsubsection{Degeneration to the Torus}

Recall that a \emph{regular point} $p \in X$ is a point at which $\omega$ has neither a zero nor a pole, and a \emph{regular trajectory} is one that does not pass through any singularities (zeros or poles) of $\omega$, but we permit it to pass through punctures at which $\omega$ does not have singularities when continued across those punctures.

To clarify potential confusion with the covering maps that will occur throughout the rest of this paper, we make the following observation.  Though it is true that if a sequence of torus covers of fixed genus converges to a connected surface of the same genus $(X', \omega')$ with covering map $\pi'_{opt}$, then $\pi'_{opt}$ is an optimal covering map, this is not necessarily true if the genus of the covering decreases.  We use the notation $\tilde{\pi}'_{opt}: (X', \omega') \rightarrow E$ for the limiting covering map even though it is not optimal, and let $\pi'_{opt}: (X', \omega') \rightarrow E$ denote the optimal covering map (with no intermediate torus cover).  For example, if $(X', \omega')$ is the EW, then $\pi'_{opt}$ is the double cover of the torus branched over all four $2$-torsion points, and if $(X', \omega')$ is a punctured torus, then $\pi'_{opt}$ is the identity map.

We say that a node $p$ of a degenerate surface \emph{touches} a sheet $S$ of a torus cover, where $\pi_{opt}(S) = E$ is injective, if for all $\varepsilon > 0$, there exists a path $\gamma: [0,1] \rightarrow (X, \omega)$ of length $\varepsilon$ such that $\gamma(0) = p$ and $\gamma(1) \in S$.  For example, if $\mathcal{P}$ is a connected set of punctures containing exactly one point $p$, and $p$ is at a simple zero of $(X', \omega')$, then $p$ touches two sheets.  Similarly, if $\mathcal{P}$ is a connected set of punctures containing exactly two points $p$ and $q$, and $p$ and $q$ are at regular points of $(X', \omega')$, then $p$ and $q$ touch two sheets.

\begin{lemma}
\label{NoIsolatedAtRegPts}
Let $\cM$ be an AIS with completely degenerate KZ-spectrum.  Let $(X',\omega')$ be a degenerate surface carrying a holomorphic Abelian differential in the boundary of $\cM$.  If $p$ is a puncture at a regular point of $\omega'$ and $\mathcal{P}$ is the connected set of punctures containing $p$, then $\mathcal{P} \not= \{p\}$.
\end{lemma}

\begin{proof}
By contradiction, assume that $\mathcal{P} = \{p\}$.  Since both $(X,\omega)$ and $(X', \omega')$ are torus covers, a degeneration on $(X,\omega)$ occurs only if branch points on the torus covered by $(X, \omega)$ collide.  Moreover, if branch points collide, then a degeneration only occurs if the zeros of $\omega$ over those branch points collide.  However, by definition of a ramified covering, there is more than one sheet that touches a zero of $\omega$.  Let $z$ be a zero of $\omega$ that converges to $p$.  Let $T_1$ and $T_2$ be two sheets that map to the torus under $\pi_{opt}$ that touch $z$.  After degenerating to $(X', \omega')$, $p$ must lie on exactly one of those sheets, say $T_1$ without loss of generality.  However, a local neighborhood of the node that was removed to yield $\mathcal{P}$ had non-trivial intersection with both $T_1$ and $T_2$ before the degeneration.  Since $p$ does not lie at a zero of $\omega'$ by assumption, there must also be a puncture $p' \in T_2$ such that $p' \in \mathcal{P}$.  This contradiction proves the lemma.
\end{proof}

Since the intersection of all closed curves through a point on a torus is exactly that point, and a non-zero holomorphic differential on the torus has no zeros, the following corollary follows from Lemmas \ref{ConnectedPctsAtIntersect} and \ref{NoIsolatedAtRegPts}.

\begin{corollary}
\label{NoDegToTorus}
Let $\cM$ be an AIS with completely degenerate KZ-spectrum.  Let $(X',\omega')$ be a degenerate surface carrying a holomorphic Abelian differential in the boundary of $\cM$.  If we remove any nodes of $X'$ resulting from curves collapsing and continue the differential on $X'$ holomorphically across all punctures to yield a surface $\overline{X'}$ with no punctures, then $\overline{X'}$ is not a torus.
\end{corollary}

\subsubsection{Degeneration to the EW}

We continue by addressing the possible ways a genus five surface can degenerate to a genus three or genus four surface.

\begin{lemma}
\label{EWConnSetPctsLem}
Let $\cM$ be an AIS with completely degenerate KZ-spectrum.  Let $(X',\omega')$ be a degenerate surface carrying a holomorphic Abelian differential in the boundary of $\cM$ such that removing the nodes of $X'$, and filling every resulting puncture yields a surface in the Teichm\"uller curve of the EW.  If $\mathcal{P}$ is a connected set of punctures on $(X', \omega')$, then either 
\begin{itemize}
\item $\mathcal{P}$ contains exactly one point lying at a simple zero of $\omega'$, or
\item $\mathcal{P}$ contains exactly two points, both of which lie at regular points of $\omega'$.
\end{itemize}
\end{lemma}

\begin{proof}
Recall that the EW is a degree two ramified covering of a torus, so $\mathcal{P}$ contains at most two elements.  Since every point on the EW is either a regular point of $\omega'$ or a simple zero of $\omega'$, if $\mathcal{P}$ contains exactly one point, then by Lemma \ref{NoIsolatedAtRegPts}, it must lie at a simple zero.

Let $\tilde{\pi}_{opt}$ denote the covering map of the EW to the torus $E$.  If $p \in E$ and $\tilde{\pi}_{opt}^{-1}(p)$ contains exactly two points, then both lie at regular points of $\omega'$ as claimed.
\end{proof}

\begin{lemma}
\label{SheetsBdZeroOrds}
Let $\cM$ be a family of torus covers.  Let $(X',\omega')$ be a degenerate surface carrying a holomorphic Abelian differential in the boundary of $\cM$, and $\mathcal{P}$ be a connected set of punctures on $(X', \omega')$.  If $\mathcal{P}$ touches $n$ sheets and $z \in (X, \omega) \in \cM$ is a zero of $\omega$ that converges to at element of $\mathcal{P}$ on $(X', \omega')$, then $z$ has order at most $n-1$.
\end{lemma}

\begin{proof}
Observe that for any torus cover and $m \geq 1$, a zero of order $m$ in the cover touches $m + 1$ sheets of the cover.  Recall that as we converge from surfaces $(X, \omega) \in \cM$ to $(X', \omega')$, zeros of $\omega$ collide and curves pinch, only if the branch points on the underlying torus converge.  Trivially, each zero of $\omega$ lies exactly over a single branch point of the underlying torus.  In particular, the order of each zero is at most one less than the degree of the torus cover.  By the same reasoning, if collapsing a collection of zeros results in a node, which when removed results in the connected pair of punctures $\mathcal{P}$, then the number of sheets touching that node must bound the order of the zeros converging to that node.  Hence, the lemma follows.
\end{proof}

\begin{corollary}
\label{NoDegToEW}
Let $\cM$ be an AIS with completely degenerate KZ-spectrum.  Let $(X',\omega')$ be a degenerate surface carrying a holomorphic Abelian differential in the boundary of $\cM$.  If we remove any nodes of $X'$ resulting from curves collapsing and continue the differential on $X'$ holomorphically across all punctures to yield a surface $\overline{X'}$ with no punctures, then $(\overline{X'}, \omega')$ is not in the Teichm\"uller curve of the EW.
\end{corollary}

\begin{proof}
By contradiction, assume that $(\overline{X'}, \omega')$ is the EW.  By Lemma \ref{EWConnSetPctsLem}, every connected set of punctures touches at most two sheets of $(X', \omega')$.  Furthermore, it is trivial that if a collection of zeros collide and no curve pinches, then the resulting point is a zero whose order is equal to the sum of the zeros that collided.  Hence, any of the four simple zeros of $\omega'$ that do not lie at a puncture of $X'$, must have been a simple zero in a neighborhood of $(X', \omega')$ in $\cM$.  This implies that $\cM$ must lie in the principal stratum $\cH(1^8)$ in genus five because there are no higher order zeros of $\omega'$, and Lemma \ref{SheetsBdZeroOrds} implies that the zeros that do converge must be simple.

However, by Corollary \ref{DenseSubset} and the definition of the Forni subspace, $\cM$ contains a Teichm\"uller curve with completely degenerate KZ-spectrum.  By \cite[Cor. 5.15]{MollerShimuraTeich}, there is no Teichm\"uller curve with completely degenerate KZ-spectrum in the principal stratum in genus five.  This contradiction concludes the proof of the corollary.
\end{proof}


\subsubsection{Degeneration to the Ornithorynque}

\begin{lemma}
\label{OConnSetPctsLem}
Let $\cM$ be an AIS with completely degenerate KZ-spectrum.  Let $(X',\omega')$ be a degenerate surface carrying a holomorphic Abelian differential in the boundary of $\cM$ such that removing the nodes of $X'$, and filling every resulting puncture yields a surface in the Teichm\"uller curve of the Ornithorynque.  Then $(X', \omega')$ has exactly one connected set of punctures $\mathcal{P}$, and either 
\begin{itemize}
\item $\mathcal{P}$ contains exactly one point lying at a double zero of $\omega'$, or
\item $\mathcal{P}$ contains exactly two points, both of which lie at regular points of $\omega'$.
\end{itemize}
\end{lemma}

\begin{proof}
First observe that a connected set of punctures always results from the loss of at least one in genus.  Since $(X, \omega) \in \cM$ has genus five and $X'$ has genus four, there is exactly one connected set of punctures.

Recall that the Ornithorynque is a degree three branched covering of a torus, so $\mathcal{P}$ contains at most three elements.  Since every point on the Ornithorynque is either a regular point of $\omega'$ or a double zero of $\omega'$, if $\mathcal{P}$ contains exactly one point, then by Lemma \ref{NoIsolatedAtRegPts}, it must lie at a double zero.

Let $\pi'_{opt}$ denote the ramified covering map of the Ornithorynque to the torus $E$.  If $p \in E$ and $\mathcal{P} \subset \pi_{opt}^{\prime -1}(p)$ contains exactly two or three points, then both lie at regular points of $\omega'$.  However, if $\mathcal{P}$ contains three points, then this implies the pinching of at least two non-homologous curves, which implies that $(X, \omega) \in \cM$ has genus at least six.  Hence, the lemma follows.
\end{proof}

\begin{lemma}
\label{DegToOImpH2t4}
Let $\cM$ be an AIS with completely degenerate KZ-spectrum.  Let $(X',\omega')$ be a degenerate surface carrying a holomorphic Abelian differential in the boundary of $\cM$ such that removing the nodes of $X'$, and filling every resulting puncture yields a surface in the Teichm\"uller curve of the Ornithorynque.  Then $\cM \subset \cH(2^4)$.
\end{lemma}

\begin{proof}
By Lemmas \ref{OConnSetPctsLem} and \ref{SheetsBdZeroOrds}, $\cM$ lies in a stratum in genus five with zeros of order at most two because the Ornithorynque is a triple ramified cover of the torus.  By \cite[Cor. 5.15]{MollerShimuraTeich}, $\cM$ does not lie in the principal stratum because there are no Teichm\"uller curves with completely degenerate KZ-spectrum in that stratum.  Therefore, $\cM$ lies in a stratum with at least one double zero.

We refer to Figure \ref{OrniFig} throughout this proof.  Observe that for each of the three zeros on the Ornithorynque, there are infinitely many choices of directions such that there is a saddle connection from a zero to itself.  (This is a consequence of the fact that only three of the four $2$-torsion points of the torus are ramification points of the covering map.)

By contradiction, assume that $\cM \subset \cH(2, 1^6)$ lies in the stratum with six simple zeros.  Either the double zero converges to the connected set of punctures, or it does not.  If the double zero converges to the connected set of punctures, then each of the other double zeros on $(X',\omega')$ do not lie at punctures, but are formed by the collision of two simple zeros.  Let $v_1$ be one such double zero resulting from the collision of two simple zeros.  By \splin , let $v_1$ lie where it does in Figure \ref{OrniFig}.  In a neighborhood in $\cM$ of $(X', \omega')$, $v_1$ splits into two simple zeros.  The existence of infinitely many directions ``isolating'' $v_1$ yields the existence of a trajectory from one of the simple zeros to itself.  This contradicts \cite[Lem. 5.6]{MollerShimuraTeich} and leaves us with the assumption that the double zero does not converge to the connected set of punctures.  In this case the double zero $v_1$ is one of the double zeros on $(X', \omega')$, in the sense that there is neighborhood of $v_1$ prelimit with no other zero.  Let $v_2$ be a double zero on $(X', \omega')$ which is not at a puncture of $X'$.  Then by the assumption that $\cM \subset \cH(2, 1^6)$, $v_2$ is formed by the collision of two simple zeros.  Again we observe that on the Ornithorynque, there are infinitely many directions from $v_1$ to $v_2$.  This allows us to get the same contradiction as above because at most one of the two simple zeros can pass through $v_2$ before returning to itself.  Hence, $\cM$ lies in a stratum with at most four simple zeros, and at least two double zeros.

Since there are simple zeros on $(X, \omega) \in \cM$, but no simple zeros in the limit $(X', \omega')$, all of them must collide.  On the other hand, \cite[Lem. 5.18]{MollerShimuraTeich}, implies that if there are simple zeros, then they must lie over the $2$-torsion points of $E$ under the map $\pi'_{opt}$, which implies that they remain bounded distance away from each other for all time.\footnote{In fact, the Case i) assumption combined with \cite[Lem. 5.18]{MollerShimuraTeich} actually implies that in the stratum $\cH(2^3, 1^2)$ there are no Teichm\"uller curves with completely degenerate KZ-spectrum satisfying Case i) because it is a contradiction that there are only two odd order zeros and each of the four $2$-torsion points has an odd order zero over it.  We do not use this fact here, but highlight it to avoid confusion from such a blatant contradiction.}  Hence, the lemma follows.
\end{proof}

\begin{corollary}
\label{NoDegToO}
Let $\cM$ be an AIS with completely degenerate KZ-spectrum.  Let $(X',\omega')$ be a degenerate surface carrying a holomorphic Abelian differential in the boundary of $\cM$.  If we remove any nodes of $X'$ resulting from curves collapsing and continue the differential on $X'$ holomorphically across all punctures to yield a surface $\overline{X'}$ with no punctures, then $(\overline{X'}, \omega')$ is not in the Teichm\"uller curve of the Ornithorynque.
\end{corollary}

\begin{proof}
By Lemma \ref{DegToOImpH2t4}, $\cM \subset \cH(2^4)$.  Therefore the degeneration is given by exactly two double zeros colliding to form a single double zero $v_1$ and the connected set of punctures is exactly $\mathcal{P} = \{v_1\}$.  We claim that if $(X, \omega) \in \cM$ is in a sufficiently small neighborhood of $(X', \omega')$, then it can be realized as the surface in Figure \ref{OrniNbhdFig}.  Let $v'$ be the double zero converging to $v_1$, so in particular the distance between $v_1$ and $v'$ can be taken arbitrarily small.  We are careful to distinguish between the covering maps $\tilde{\pi}'_{opt}: (X', \omega') \rightarrow E$ and $\pi'_{opt}: (X', \omega') \rightarrow E'$ because they have degree nine (by \cite[Cor. 5.15]{MollerShimuraTeich}) and three, respectively.

Since the covering map $\pi_{opt}$ maps to a fixed torus (forgetting punctures) for all $(X, \omega)$ and $(X, \omega)$ converges to $(X', \omega')$, there exists an unramified triple covering $\pi: E' \rightarrow E$.  Hence, each of the three sheets of $\pi'_{opt}$ (which map to $E'$) in Figure \ref{OrniNbhdFig} can be cut into three sheets of $\tilde{\pi}'_{opt}$, each of which map to $E$, and the location of the zeros differ by at most the set of points that map to the same point under $\pi_{opt}$.  However, after considering a sufficiently small neighborhood $B \subset \overline{\cM}$ of $(X', \omega')$, the zeros remain on the same sheets of $\pi_{opt}: (X, \omega) \rightarrow E$ for all $(X, \omega) \in B$.  Since the limit $(X', \omega')$ can be realized as a triple cover, the surfaces $(X, \omega) \in B$ can also be realized as a triple cover because, of the nine points in $\pi_{opt}^{-1}(\pi_{opt}(v'))$, only three have distance $\varepsilon$ from copies of $v_1$.  This implies that $(X, \omega)$ admits a triple cover to the torus, which contradicts \cite[Cor. 5.15]{MollerShimuraTeich}.

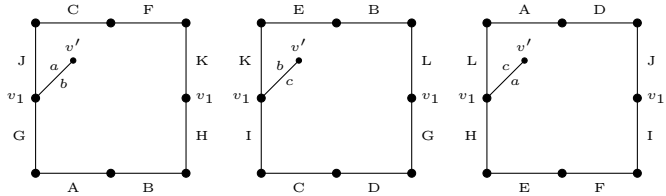
\begin{figure}[htb]
\centering
\begin{minipage}{0.4\linewidth}
\centering
\begin{tikzpicture}[scale=0.50]
\draw (-8,0) -- (-8,4) -- (-4,4) -- (-4,0) -- cycle;
\draw (-2,0) -- (-2,4) -- (2,4) -- (2,0) -- cycle;
\draw (4,0) -- (4,4) -- (8,4) -- (8,0) -- cycle;
\foreach \x in {(-8,0),(-8,2),(-8,4),(-6,0),(-6,4),(-4,4),(-4,2),(-4,0)} \filldraw[fill=black] \x circle (3pt);
\foreach \x in {(-2,0),(-2,2),(-2,4),(0,0),(0,4),(2,4),(2,2),(2,0)} \filldraw[fill=black] \x circle (3pt);
\foreach \x in {(4,0),(4,2),(4,4),(6,0),(6,4),(8,4),(8,2),(8,0)} \filldraw[fill=black] \x circle (3pt);
\draw (-7,4) node[above] {\tiny C};
\draw (-5,4) node[above] {\tiny F};
\draw (-7,0) node[below] {\tiny A};
\draw (-5,0) node[below] {\tiny B};
\draw (-8,1) node[left] {\tiny G};
\draw (-8,3) node[left] {\tiny J};
\draw (-4,1) node[right] {\tiny H};
\draw (-4,3) node[right] {\tiny K};

\draw (-1,4) node[above] {\tiny E};
\draw (1,4) node[above] {\tiny B};
\draw (-1,0) node[below] {\tiny C};
\draw (1,0) node[below] {\tiny D};
\draw (-2,1) node[left] {\tiny I};
\draw (-2,3) node[left] {\tiny K};
\draw (2,1) node[right] {\tiny G};
\draw (2,3) node[right] {\tiny L};

\draw (5,4) node[above] {\tiny A};
\draw (7,4) node[above] {\tiny D};
\draw (5,0) node[below] {\tiny E};
\draw (7,0) node[below] {\tiny F};
\draw (4,1) node[left] {\tiny H};
\draw (4,3) node[left] {\tiny L};
\draw (8,1) node[right] {\tiny I};
\draw (8,3) node[right] {\tiny J};

\draw (-8,2) node[left] {\tiny $v_1$};
\draw (-4,2) node[right] {\tiny $v_1$};
\draw (-2,2) node[left] {\tiny $v_1$};
\draw (2,2) node[right] {\tiny $v_1$};
\draw (4,2) node[left] {\tiny $v_1$};
\draw (8,2) node[right] {\tiny $v_1$};

\draw (-8,2) -- (-7,3);
\filldraw[fill=black] (-7,3) circle (2pt);
\draw (-7,3) node[above] {\tiny $v'$};
\draw (-7.5,2.5) node[above] {\tiny $a$};
\draw (-7.25,2.75) node[below] {\tiny $b$};

\draw (-2,2) -- (-1,3);
\filldraw[fill=black] (-1,3) circle (2pt);
\draw (-1,3) node[above] {\tiny $v'$};
\draw (-1.5,2.5) node[above] {\tiny $b$};
\draw (-1.25,2.75) node[below] {\tiny $c$};

\draw (4,2) -- (5,3);
\filldraw[fill=black] (5,3) circle (2pt);
\draw (5,3) node[above] {\tiny $v'$};
\draw (4.5,2.5) node[above] {\tiny $c$};
\draw (4.75,2.75) node[below] {\tiny $a$};
\end{tikzpicture}
\end{minipage}
\caption{A Neighborhood of the Ornithorynque in $\cH(2^4)$}
\label{OrniNbhdFig}
\end{figure}

Alternatively, a more explicit contradiction can be drawn from Figure \ref{OrniNbhdFig}.  Let $\varepsilon > 0$ denote the distance from $v_1$ to $v'$.  Then after acting by a combination of upper or lower triangular matrices, if necessary, it suffices to assume that $v'$ lies above $v_1$ as in Figure \ref{OrniNbhdFig}.  There is a unique choice for the identification of the slits from $v_1$ to $v'$ such that $v_1$ and $v'$ are double zeros, and $v_1$ remains a double zero after $v'$ converges to $v_1$.  By inspection, the horizontal trajectory from $b$ to $b$ is strictly less than the length of a parallel trajectory, which contradicts Corollary \ref{RankOneCylConfig} and completes the proof.
\end{proof}

\subsubsection{Proof of the Proposition}

\begin{proof}[Proof of Proposition \ref{ZerosDegSameGen}]
Combining Lemma \ref{DegSurfExists} and Corollary \ref{DegSurfsCor}, we get the existence of a degenerate surface $(X', \omega')$ in the boundary of $\cM$ with completely degenerate KZ-spectrum, such that $X'$ has genus one, three, four, or five.  By Corollaries \ref{NoDegToTorus}, \ref{NoDegToEW}, and \ref{NoDegToO}, $X'$ has genus five.
\end{proof}

\bibliography{fullbibliotex}{}

\end{document}